\documentclass{amsart}

\usepackage{amsmath}
\usepackage{amsthm}
\usepackage{amssymb}
\usepackage{xspace}
\usepackage{enumerate}

\theoremstyle{plain}
\newtheorem{theorem}{Theorem}
\newtheorem{lemma}[theorem]{Lemma}

\newcommand{\rea}{\ensuremath{\mathbf{R}}\xspace}
\newcommand{\ints}{\ensuremath{\mathbf{Z}}\xspace}
\newcommand{\intd}[1]{\,\mathrm{d}#1 \,}

\newcommand{\ceil}[1]{\left\lceil #1 \right\rceil}
\newcommand{\closure}[1]{\ensuremath{\overline{#1}}\xspace}
\newcommand{\leb}{\ensuremath{\mathcal{L}}\xspace}
\newcommand{\E}{\ensuremath{\mathbf{E}}\xspace}
\newcommand{\myvec}[1]{\ensuremath{\overline{#1}}\xspace}
\newcommand{\Dset}{\ensuremath{\mathcal{U}}\xspace}
\newcommand{\hmeas}{\ensuremath{\mathcal{H}}\xspace}
\newcommand{\dimh}{\ensuremath{\dim_{\mathrm{H}}}\xspace}
\newcommand{\dimf}{\ensuremath{\dim_{\mathrm{F}}}\xspace}

\newcommand\restr[2]{{
  \left.\kern-\nulldelimiterspace
  #1
  \vphantom{|}
  \right|_{#2}
  }}

\DeclareMathOperator{\supp}{supp}
\DeclareMathOperator{\e}{e}
\DeclareMathOperator{\conv}{conv}
\DeclareMathOperator{\dist}{dist}


\begin{document}
\title{Fourier dimension of random images}
\author{Fredrik Ekstr\"om}
\address{Centre for Mathematical Sciences, Lund University,
Box 118, 221 00 Lund, Sweden}
\email{fredrike@maths.lth.se}

\begin{abstract}
Given a compact set of real numbers, a random $C^{m + \alpha}$-diffeo\-morph\-ism is
constructed such that the image of any measure concentrated on the set 
and satisfying a certain condition involving a real number $s$, almost 
surely has Fourier dimension greater than or equal to $s / (m + \alpha)$.
This is used to show that every Borel subset of the real numbers of Hausdorff
dimension $s$ is $C^{m + \alpha}$-equivalent
to a set of Fourier dimension greater than or equal to $s / (m + \alpha)$. In particular every Borel set is
diffeomorphic to a Salem set, and the Fourier dimension is not invariant under $C^m$-diffeomorphisms
for any $m$.
\end{abstract}

\maketitle

\section{Introduction}
The \emph{Fourier dimension} of a Borel probability measure on
$\rea^d$ measures the polynomial rate of decay of the Fourier
transform of $\mu$, and is defined to be the supremum of all
$s$ in $[0, d]$ such that $|\widehat\mu(\xi)| |\xi|^{s / 2}$
is bounded. The Fourier dimension of a Borel set $F$ is the
supremum of the Fourier dimensions of all probability measures
that give full measure to $F$. It can be shown that if the
Fourier dimension of $\mu$ is greater than $s$ then
	$$
	\iint |y - x|^{-s} \intd\mu(x) \intd\mu(y) < \infty
	$$
(see \cite[Lemma~12.12]{mattila95}), and it follows that the Hausdorff
dimension of $F$ is always greater than or equal to the
Fourier dimension of $F$. If the Fourier and
Hausdorff dimensions of $F$ are equal, then $F$ is called
a \emph{Salem set}. An example of a set that is not a
Salem set is the ternary Cantor set, which has Fourier
dimension $0$. (More generally, a compact proper subset
of $[0, 1]$ that is invariant under multiplication mod $1$ by
an integer greater than or equal to $2$ cannot support a measure
whose Fourier transform tends to $0$ at infinity.)

This paper gives a construction of a random $C^{m + \alpha}$-diffeomorphism
$f_\omega: \rea \to \rea$, given a compact subset $E$ of \rea.
If $\lambda$ is a probability measure on $E$ such that
the number of connected components of $E^c$ that are included
in any given interval $J$ and have length at least $x^{-1}$ is bounded below (up to constants)
by $\log \lambda(J) + s\log x$, then the image of $\lambda$
under $f_\omega$ almost surely has Fourier dimension greater than or equal to
$s / (m + \alpha)$. The mentioned condition is satisfied for example if $E$
is the attractor of a ``nice'' iterated function system and $\lambda$
is a probability measure on $E$ such that $\lambda(I) \leq A|I|^s$
for every interval $I$.

Taking $E = C$ in the construction of $f_\omega$, where $C$ is a certain fat
Cantor set, it is shown that for any Borel set $F$ of positive
$s$-dimensional Hausdorff measure there is a real number $t$ such that
almost surely $\dimf f_\omega(C \cap (F + t)) \geq s / (m + \alpha)$.
This is used to prove that for any Borel set $F$ there is a
$C^{m + \alpha}$-diffeomorphism $f$ such that $\dimf f(F) \geq s / (m + \alpha)$,
where $s$ is the Hausdorff dimension of $F$ (even if the $s$-dimensional
Hausdorff measure of $F$ is $0$). In particular, every
Borel subset of \rea is diffeomorphic to a Salem set, and
the Fourier dimension is not $C^m$-invariant for any $m$.

A remaining question is 
whether there exists a Borel set $F$ 
such that $\dimf f(F) \leq \dimh F / (m + \alpha)$ for
every $C^{m + \alpha}$-diffeomorphism $f$. One might also ask
whether the Fourier dimension is invariant under $C^\infty$-diffeomorphisms,
since the statements that are proved here become empty when $m \to \infty$.
It follows from previous work (see below) that this is not the case for
subsets of $\rea^2$, and not for subsets of \rea if non-invertible
$C^\infty$-functions are considered.

\subsection{Related work}
It was shown by Salem in \cite{salem51} that there exist Salem subsets
of \rea of any dimension between $0$ and $1$, using a construction of
a Cantor set where the contraction ratios are chosen randomly.

The one-dimensional Brownian motion is almost surely H\"older continuous
with any exponent less than $1 / 2$, and more generally the fractional Brownian
motion with Hurst index $\alpha$ is H\"older continuous with any exponent less
than $\alpha$ for $\alpha \in (0, 1)$. It was shown by Kahane that
if $E$ is any compact subset of \rea  of Hausdorff dimension $s \leq \alpha$,
then the image of $E$ under fractional Brownian motion with Hurst index $\alpha$
is almost surely a Salem set of dimension $s / \alpha$
(see \cite[Chapters~17 and 18]{kahane85}. In fact Kahane proved a more general
statement, which allows both $E$ and the image of $E$ to lie in higher dimensional
Euclidean spaces).
% p. 251 and p. 267

In \cite{bluhm99}, Bluhm gave a method for randomly perturbing a class 
of self-similar measures on $\rea^d$, such that the perturbed measure
almost surely has Fourier dimension equal to the similarity dimension
of the original measure. For $d = 1$, the uniform measures on Cantor
sets with constant contraction ratio are among the measures considered
by Bluhm, and if the parameters in the construction are chosen suitably
then the perturbation is a bi-Lipschitz map. Thus it follows
from Bluhm's result that such Cantor sets are bi-Lipschitz equivalent to
Salem sets.

An explicit example of a Salem set is given by the set of $\alpha$-badly
approximable numbers, that is, the set
	$$
	E(\alpha) =
	\bigcap_{n = 1}^\infty \bigcup_{k = n}^\infty \left\{
	x \in [0, 1]; \, \| kx \| \leq k^{-(1 + \alpha)} \right\},
	$$
where $\| \cdot \|$ denotes the distance to the nearest integer. By a theorem
of Jarn\'{\i}k \cite{jarnik29} and Besicovitch \cite{besicovitch34} the set
$E(\alpha)$ has Hausdorff dimension $2 / (2 + \alpha)$ for $\alpha > 0$, and
Kaufman \cite{kaufman81} showed that there is a probability measure on $E(\alpha)$ with
Fourier dimension $2 / (2 + \alpha)$ (see also Bluhm's paper \cite{bluhm98}).

It follows from a result of Kaufman \cite{kaufman76} that any $C^2$-curve in $\rea^2$
with non-zero curvature has Fourier dimension $1$. Since line segments in $\rea^2$
have Fourier dimension $0$ this shows that the Fourier dimension in $\rea^2$ is
not in general invariant under $C^\infty$-diffeomorphisms.

In \cite{ekstrom}, subsets $A$ and $B$ of \rea were constructed
such that
	$$
	\max\left(\dimf A, \, \dimf B\right) < 1 \qquad \text{and} \qquad
	\dimf A \cup B = 1,
	$$
and they can be taken to be included in $[1, 2]$. If $f(x) = x^2$ then
	$$
	\dimf f\left(\sqrt{A} \cup \left(-\sqrt{B}\right)\right) = \dimf A \cup B = 1
	$$
and
	$$
	\dimf \left(\sqrt{A} \cup \left(-\sqrt{B}\right)\right) =
	\max\left(\dimf \sqrt{A}, \, \dimf \sqrt{B}\right)
	$$
since $\sqrt{A}$ and $-\sqrt{B}$ are separated (see \cite[Theorem~2]{EPS15}).
Thus $f$ changes the Fourier dimension of at least one of $\sqrt{A}$,
$\sqrt{B}$ and $\sqrt{A} \cup (-\sqrt{B})$, showing that the Fourier dimension
in \rea is not in general invariant under $C^\infty$-functions.

\subsection{Some definitions and notation}
The \emph{Fourier transform} of a probability measure $\mu$ on $\rea^d$
is defined for $\xi \in \rea^d$ by
	$$
	\widehat\mu(\xi) = \int \e(- \xi \cdot x) \intd\mu (x),
	$$
where $\e(y) = e^{2\pi i y}$ and $\cdot$ is the Euclidean inner product.
If $f$ and $g$ are complex-valued functions with the same domain, then
	$$
	f(\xi) \lesssim g(\xi)
	$$
means that there is a constant $C$ such that
	$$
	|f(\xi)| \leq C |g(\xi)|
	$$
for all $\xi$ in the domain. The Fourier dimension of $\mu$ is then
defined by
	$$
	\dimf \mu = \sup \left\{s\in[0, d]; \, \widehat\mu(\xi) \lesssim |\xi|^{-s/2} \right\},
	$$
and the Fourier dimension of a Borel set $F$ is defined by
	$$
	\dimf F = \sup\left\{ \dimf \mu; \,
	\mu \text{ is a probability measure on $\rea^d$ and }
	\mu(F) = 1 \right\}.
	$$

If $\mu$ is a measure on a measurable space $X$ and $f: X \to Y$ is a measurable function,
then $f\mu$ denotes the transportation of $\mu$ by $f$, namely, the measure on $Y$ defined
by
	$$
	(f\mu)(A) = \mu(f^{-1}(A)).
	$$
The formula for change of variable in Lebesgue integrals thus reads
	$$
	\int_Y g \intd{(f\mu)} = \int_X g \circ f \intd{\mu},
	$$
where $g$ is a real- or complex-valued function on $Y$.

Any increasing function
$\varepsilon: [0, \infty) \to [0, \infty)$ such that
	$$
	\varepsilon(0) = \lim_{t \to 0} \varepsilon(t) = 0
	$$
is called a \emph{modulus of continuity}. A function $g: \rea \to \rea$
is uniformly continuous with modulus $\varepsilon$ if
	$$
	\left| g(y) - g(x) \right| \leq \varepsilon(|y - x|)
	$$
for all $x, y \in \rea$, and a set $G$ of functions
$\rea \to \rea$ is \emph{uniformly equicontinuous} with modulus $\varepsilon$
if every $g \in G$ is uniformly continuous with modulus $\varepsilon$.
A function is \emph{H\"older continuous} with exponent $\alpha \in (0, 1]$
if it is uniformly continuous with modulus $\varepsilon(t) = C t^\alpha$
for some constant $C$.

\section{Constructions and main results} \label{section2}
This section contains the main constructions and the statements of the
main results. Proofs are given in later sections.

\begin{theorem} \label{endthm1}
Let $F$ be a Borel subset of \rea. Then there exists a
$C^{m + \alpha}$-diffeomorphism $f: \rea \to \rea$ such that
	$$
	\dimf f(F) \geq \frac{\dimh F}{m + \alpha}.
	$$
\end{theorem}

The proof of Theorem~\ref{endthm1} is based on Theorem~\ref{equicontthm}
and Theorem~\ref{thm2} below, and is given at the end of Section~\ref{lastsection}.

\subsection{The random map $f_\omega$ on a compact set $E$}
Given a compact subset $E$ of \rea, let $D$ be the set of bounded connected
components of $E^c$. Choose non-negative numbers $\{ \delta_\Dset \}_{\Dset \in D}$
whose sum is finite
and let $\Omega = \text{\Large{$\times$}}_{\Dset \in D} [0, \delta_\Dset]$.
For $\omega \in \Omega$, define $f_\omega$ on $E$ by
	$$
	f_\omega(x) = x + \sum_{\Dset \subset (-\infty, x)} \omega_\Dset,
	$$
where the sum is over those $\Dset \in D$ that lie to the left
of $x$ (thus $f_\omega$ can be thought of as increasing the size of
each hole \Dset in $E$ by an additive amount of $\omega_\Dset$).
Fix some probability measure $\nu$ on $[0, 1]$ such that
$\lim_{|\xi| \to \infty} \widehat\nu(\xi) = 0$ and let
$\nu_\Dset$ be the image of $\nu$ under the map $x \mapsto \delta_\Dset x$.
Let $P$ be the product measure on $\Omega$ that projects to $\nu_\Dset$
on the $\Dset$-coordinate.

If $J$ is an interval and $x > 0$, let
	$$
	\psi(J, x) = \# \left\{ \Dset \in D;
	\, \Dset \subset J \text{ and } \delta_\Dset \geq x^{-1} \right\}.
	$$

\begin{theorem} \label{thm1}
Let $s \in [0, 1]$ and let $\lambda$ be a probability measure on $E$.
Suppose that there are constants $a$ and $b > 0$ and some $x_0$ such that
	$$
	\psi(J, x) \geq a + b \left(
	\log \lambda(J) + s \log x \right)
	$$
for every interval $J$ and every $x \geq x_0$. Then almost
surely $\dimf f_\omega \lambda \geq s$.
\end{theorem}

\subsection{Extending $f_\omega$ to \rea} \label{extsec}
Let $\varphi$ be an increasing $C^\infty$-function on \rea
that is $0$ on $(-\infty, 0]$ and $1$ on $[1, \infty)$. Then
	$$
	f_\omega(x) = x + \sum_{\Dset \in D} \omega_\Dset
	\varphi\left( \frac{x - \inf \Dset}{|\Dset|} \right)
	$$
is an extension of $f_\omega$ to all of \rea. Let $m$ be
a positive integer and let $\alpha \in [0, 1]$. From now on 
choose $\delta_\Dset = |\Dset|^m \delta(|\Dset|)$, where
	$$
	\delta(t) =
	\begin{cases}
	\frac{1}{\max\left( -\log t, \log 2 \right)} &
	\text{ if } \alpha = 0 \\
	t^\alpha &
	\text{ if } \alpha \in (0, 1].
	\end{cases}
	$$

\begin{theorem}	\label{equicontthm}
The function $f_\omega$ is a $C^{m + \alpha}$-diffeomorphism for every
$\omega \in \Omega$, and $\{f_\omega^{(m)}\}_{\omega \in \Omega}$
is uniformly equicontinuous with modulus
$2 \| \varphi^{(m + 1)} \|_\infty \, \delta$.
\end{theorem}

\subsection{Consequences of Theorem~\ref{thm1}}
\begin{theorem}	\label{iterthm}
Let $\Phi = \{ F_1, \ldots, F_N \}$ be an iterated function system
on \rea consisting of contracting $C^{1 + \beta}$ diffeomorphisms.
Take $E$ to be the attractor of $\Phi$ and assume that the interiors
of the convex hulls of $F_i(E)$ and $F_j(E)$ are disjoint whenever $i \neq j$.
Let $\lambda$ be a probability measure on $E$ such that $\lambda(I) \leq A |I|^s$
for every interval $I$. Then almost surely
$\dimf f_\omega \lambda \geq s / (m + \alpha)$.
\end{theorem}

If $s = \dimh E$ then there is a probability measure $\lambda$ on $E$ such that
$\lambda(I) \leq A |I|^s$ for every interval $I$ (see
\cite[Theorem~5.3]{falconer97}), and thus
$\dimf f_\omega E \geq \dimh E / (m + \alpha)$ almost surely. By using
a particular set $C$ in the construction of $f_\omega$, similar results can be obtained
for general Borel sets.

Let $(c_k)_{k = 1}^\infty$ be an increasing sequence of positive numbers that
converges to $1 / 2$ and satisfies
	$$
	\prod_{k = 1}^\infty 2c_k > 0\quad
	\text{and}\quad
	\lim_{k \to \infty} \frac{\log(1 - 2c_k)}{k} = 0
	$$
(for example $c_k = 1 / 2 - 1 / k^2$). Let $C_0 = [0, 1]$ and for $k \geq 1$ let $C_k$
be the set obtained by removing from every connected component $I$ of $C_{k - 1}$ the
open interval that is concentric with $I$ and has length $(1 - 2 c_k)|I|$. (Thus $C_k$
consists of $2^k$ closed intervals of length $\prod_{i = 1}^k c_i$.)
Let
	$$
	C = \bigcap_{k = 0}^\infty C_k;
	$$
this is a compact set of positive Lebesgue measure.

\begin{theorem}	\label{thm2}
Take $E = C$ and let $F$ be a Borel subset of $\rea$ such that
$\hmeas^s(F) > 0$. Then there is some $t \in \rea$ such that almost surely
	$$
	\dimf f_\omega(C \cap (F + t)) \geq \frac{s}{m + \alpha}.
	$$
\end{theorem}

\section{Proof of Theorem~\ref{thm1}}
The following lemma says that an almost sure bound for the Fourier dimension of a random measure can be
proved by estimating the decay rate of $\E |\widehat\mu(\xi)|^{2q}$ for large $q$. This idea
was used in the works of Salem, Kahane and Bluhm mentioned in the introduction.
The proof of Theorem~\ref{thm1} is also much inspired by Kahane's and Bluhm's proofs.

\begin{lemma} \label{kahanelemma}
Let $(\Omega, P)$ be a probability space and let $\Omega \ni \omega \mapsto \mu_\omega$
be a random probability measure on $\rea$ such that $|\supp \mu_\omega| < M$
almost surely, for some constant $M$. Suppose that
	$$
	\E\left( |\widehat\mu_\omega(\xi)|^{2q} \right)
	\lesssim |\xi|^{-sq + 1}
	$$
for $q = 1, 2, \ldots$. Then almost surely
	$$
	\widehat\mu_\omega(\xi) \lesssim |\xi|^{-s/2 + \varepsilon}
	$$
for every $\varepsilon > 0$.

\begin{proof}
The assumption on the decay of
$\E |\widehat\mu_\omega(\xi)|^{2q}$ implies that
	\begin{align*}
	\int \sum_{\xi \in \ints / M} |\xi|^{sq - 3}
	| \widehat\mu_\omega(\xi) |^{2q} \intd{P}(\omega)
	&=
	\sum_{\xi \in \ints / M} |\xi|^{sq - 3} \E |\widehat\mu_\omega(\xi)|^{2q} \\
	&\leq \text{const.} \times \sum_{\xi \in \ints / M} |\xi|^{-2}
	< \infty.
	\end{align*}
Thus for a.e.~$\omega$ the sum in the first expression is finite,
so
	$$
	\lim_{\substack{|\xi| \to \infty \\ \xi \in \ints / M}}
	|\xi|^{sq - 3} | \widehat\mu_\omega(\xi) |^{2q} = 0,
	$$
and in particular
	$$
	|\widehat\mu(\xi)| \lesssim |\xi|^{-s / 2 + 3 / (2q)}, \qquad
	\xi \in \ints / M.
	$$
It follows from a lemma of Kahane \cite[p.~252]{kahane85} that
$\widehat\mu_\omega(\xi) \lesssim |\xi|^{-s / 2 + 3 / (2q)}$ for $\xi \in \rea$ as well.
Letting $q \to \infty$ establishes the
conclusion for any fixed $\varepsilon > 0$, and letting 
$\varepsilon \to 0$ along a countable set then proves the lemma.
\end{proof}
\end{lemma}

\begin{proof}[Proof of Theorem~\ref{thm1}]
Let $\mu_\omega = f_\omega \lambda$. By Lemma~\ref{kahanelemma} it suffices to show that
	$$
	\E \left(|\widehat\mu_\omega(\xi)|^{2q} \right) \lesssim |\xi|^{-sq + 1} \qquad
	\text{for } q = 1, 2, \ldots.
	$$
Now,
	\begin{align*}
	|\widehat\mu_\omega(\xi)|^{2q} &= \left(
	\int \e(\xi(y - x)) \intd{\mu_\omega^2}(x, y) \right)^q \\
	&=
	\left( \int \e(\xi(f_\omega(y) - f_\omega(x))) \intd{\lambda^2}(x, y) \right)^q \\
	&=
	\int \e\left(\xi\sum_{i = 1}^q \left( f_\omega(y_i) - f_\omega(x_i) \right)\right)
	\intd{\lambda^{2q}}(\myvec x, \myvec y),
	\end{align*}
where $\myvec x = (x_1, \ldots, x_q)$ and $\myvec y = (y_1, \ldots, y_q)$. Let
	$$
	d_{\myvec x, \myvec y} =
	\sum_{i = 1}^q \left( y_i - x_i \right)
	$$
and
	$$
	h_{\myvec x, \myvec y}(z) = \#\{i; \, z < y_i \} - \#\{i; \, z < x_i \}.
	$$
If $(\myvec x, \myvec y) \in E^{2q}$ then $h_{\myvec x, \myvec y}$ is constant on
each $\Dset$ in $D$, and
	$$
	\sum_{i = 1}^q \left( f_\omega(y_i) - f_\omega(x_i) \right)
	=
	d_{\myvec x, \myvec y} + \sum_{\Dset \in D}
	h_{\myvec x, \myvec y}(\Dset) \omega_\Dset.
	$$
Thus
	$$
	|\widehat\mu_\omega(\xi)|^{2q} =
	\int \e\left( \xi d_{\myvec x, \myvec y} \right) \e\left(\xi\sum_{\Dset \in D}
	h_{\myvec x, \myvec y}(\Dset) \omega_\Dset \right)
	\intd{\lambda^{2q}}(\myvec x, \myvec y),
	$$
and integrating over $\omega$ gives
	\begin{align*}
	\E |\widehat\mu_\omega(\xi)|^{2q} &\leq
	\int \left| \int \prod_{\Dset \in D} \e(h_{\myvec x, \myvec y}(\Dset) \xi \omega_\Dset)
	\intd{P}(\omega)
	\right| \intd{\lambda^{2q}}(\myvec x, \myvec y) \\
	&=
	\int \left| \prod_{\Dset \in D} \int \e(h_{\myvec x, \myvec y}(\Dset) \xi \omega_\Dset)
	\intd{\nu_{\Dset}}(\omega_\Dset) \right| \intd{\lambda^{2q}}(\myvec x, \myvec y) \\
	&=
	\int \prod_{\Dset \in D} |\widehat\nu(h_{\myvec x, \myvec y}(\Dset) \delta_\Dset \xi)|
	\intd{\lambda^{2q}}(\myvec x, \myvec y).
	\end{align*}

Let $B_r$ be the set of $(\myvec x, \myvec y)$ such that $\lambda(J) \leq r$ whenever $J$
is an interval and $h_{\myvec x, \myvec y}$ is non-zero everywhere on $J$.
If $(\myvec x, \myvec y) \in B_r$ then for each $x_i$ there must be some $y_j$ such that
the open interval between $x_i$ and $y_j$ has $\lambda$-measure less than or equal to
$r$, since either $x_i = y_j$ for some $j$ or $h_{\myvec x, \myvec y}$ increases by $1$
at $x_i$. Thus for each fixed $\myvec y$ there is a set of $\lambda$-measure $2qr$ that
contains all the $x_i$:s whenever $(\myvec x, \myvec y) \in B_r$. Hence
	$$
	\lambda^{2q}(B_r) =
	\int \lambda^q
	\left(\{ \myvec x; \, (\myvec x, \myvec y) \in B_r \}\right)
	\intd{\lambda^q}(\myvec y)
	\leq (2qr)^q.
	$$
If on the other hand $(\myvec x, \myvec y) \notin B_r$ then there is an interval $J$ such
that $\lambda(J) \geq r$ and $h_{\myvec x, \myvec y} \neq 0$ on $J$. Then for any $K > 0$ and
for $|\xi| \geq Kx_0$,
	\begin{align*}
	\prod_{\Dset \in D} |\widehat\nu(h_{\myvec x, \myvec y}(\Dset) \delta_\Dset \xi)|
	&\leq
	\prod_{\substack{\Dset \in D \\ \Dset \subset J}} g(\delta_\Dset |\xi|)
	\leq
	g(K)^{\psi\left(J, K^{-1}|\xi|\right)} \\
	&\leq
	g(K)^{a + b \left(\log r + s \log\left(K^{-1}|\xi|\right)\right)},
	\end{align*}
where
	$$
	g(x) = \sup_{|\xi| \geq x} |\widehat\nu(\xi)|.
	$$

Thus for any positive $r$ and $K$, and $|\xi| \geq K x_0$,
	$$
	\E |\widehat\mu_\omega(\xi)|^{2q} \leq
	(2qr)^q +
	g(K)^{a + b \left(\log r + s \log\left(K^{-1}|\xi|\right)\right)}.
	$$
In particular this holds if $r$ is chosen such that
	$$
	g(K)^{a + b \left(\log r + s \log\left(K^{-1}|\xi|\right)\right)}
	= r^q,
	$$
or equivalently,
	$$
	\log r = -s \log |\xi| \left( \frac{a}{bs\log|\xi|} + 
	\left(1 - \frac{\log K}{\log |\xi|} \right)\right)
	\frac{b \log g(K)}{b \log g(K) - q}.
	$$
For a fixed $K$ the factor in the middle converges to $1$ when $|\xi| \to \infty$,
and hence
	$$
	\E |\widehat\mu_\omega(\xi)|^{2q} \lesssim
	|\xi|^{-sq \frac{b \log g(K)}{b \log g(K) - q} + \frac{1}{2}}
	$$
for any $K$. Taking $K$ large enough shows that
$\E |\widehat\mu_\omega(\xi)|^{2q} \lesssim |\xi|^{-sq + 1}$.
\end{proof}

\section{Lemma~\ref{gensmoothlemma} and the proof of Theorem~\ref{equicontthm}}
The following lemma is used in the proof of Theorem~\ref{equicontthm}
at the end of this section, and also in the proof of Theorem~\ref{endthm1}.

\begin{lemma}	\label{gensmoothlemma}
Let $\{ V_k \}_{k = 1}^\infty$ be disjoint open intervals such that $V = \bigcup_k V_k$
is bounded and let $\{g_k\}_{k = 1}^\infty$ be increasing functions $\rea \to \rea$ that are $m$ times
differentiable ($m \geq 1$), such that $\{ g_k^{(m)} \}$ is uniformly equicontinuous
with modulus $\varepsilon$ and
	\begin{equation} \label{zerocondeq}
	g_k(\inf V_k) = 0
	\quad \text{and} \quad
	g_k' = \ldots = g_k^{(m)} = 0 \text{ on } V_k^c
	\end{equation}
for all $k$. Define $g: \rea \to \rea$ by
	$$
	g(x) = \sum_{k = 1}^\infty g_k(x).
	$$
Then $g$ is $m$ times continuously differentiable with
	$$
	g(\inf V) = 0
	\quad \text{and} \quad
	g' = \ldots = g^{(m)} = 0 \text{ on } V^c,
	$$
and $g^{(m)}$ is uniformly continuous with modulus $2\varepsilon$.

\begin{proof}
It follows from \eqref{zerocondeq} that
	\begin{equation} \label{distboundeq}
	\left| g_k^{(m)}(x) \right| \leq \varepsilon\big(\dist(x, V^c)\big)
	\end{equation}
and
	\begin{equation} \label{gkboundeq}
	g_k(\inf V_k + t) \leq \frac{\varepsilon(t)}{m!} t^m
	\qquad\text{for } t \geq 0
	\end{equation}
for all $k$. In particular,
	$$
	g(x) \leq \sum_k | g_k(V_k) | \leq
	\frac{\varepsilon\big(\sup_k |V_k|\big)}{m!} \sum_k |V_k|^m < \infty,
	$$
so that $g$ is well-defined.

It is clear that $g^{(m)}$ exists and is continuous on $V$,
and by \eqref{distboundeq},
	$$
	\lim_{\substack{x \to V^c \\ x \in V}} g^{(m)}(x) = 0.
	$$
To prove that $g$ is $m$ times continuously differentiable it
will be shown that $g^{(m)}$ exists and is $0$ on $V^c$. For
this it suffices to consider limits from the right, since
$x \mapsto -g(-x) + |g(V)|$ has the same form as $g$, with
$x \mapsto -g_k(-x) + |g_k(V_k)|$ instead of $g_k$. So
take any $x \in V^c$ and any $h \geq 0$. Then
	$$
	g(x + h) - g(h) = \sum_{V_k \subset (x, x + h)} |g_k(V_k)|
	+ \sum_{x + h \in V_n} g_n(x + h)
	$$
(the second sum has one term if $x + h \in V$ and is empty otherwise).
By \eqref{gkboundeq},
	\begin{align*}
	\sum_{V_k \subset (x, x + h)} |g_k(V_k)| &\leq
	\frac{\varepsilon(h)}{m!} \sum_{V_k \subset (x, x + h)} |V_k|^m\\
	&\leq
	\frac{\varepsilon(h)}{m!} \left(\sum_{V_k \subset (x, x + h)} |V_k|\right)^m
	\leq
	\frac{\varepsilon(h)}{m!} h^m
	\end{align*}
and
	$$
	\sum_{x + h \in V_n} g(x + h) \leq \frac{\varepsilon(h)}{m!} h^m,
	$$
and thus
	$$
	g(x + h) - g(x) \leq \frac{2 \varepsilon(h)}{m!} h^m.
	$$
Hence $g$ is $m$ times differentiable at $x$, and
	$$
	g'(x) = \ldots = g^{(m)}(x) = 0.
	$$

Finally, it will be shown that $g^{(m)}$ is uniformly continuous with
modulus $2 \varepsilon$.
If there is some $n$ such that $x$ and $y$ both lie in $\closure{V}_n$ then
	$$
	\left| g^{(m)}(y) - g^{(m)}(x) \right|
	=
	\left| g_n^{(m)}(y) - g_n^{(m)}(x) \right| \leq
	\varepsilon(| y - x |).
	$$
Otherwise the open interval between $x$ and $y$ intersects $V^c$,
so
	\begin{align*}
	\left| g^{(m)}(y) - g^{(m)}(x) \right|
	&\leq
	\left| g^{(m)}(y) \right| + \left| g^{(m)}(x) \right| \\
	&\leq \varepsilon(\dist(y, V^c)) + \varepsilon(\dist(x, V^c)))
	\leq 2 \varepsilon(|y - x|). \qedhere
	\end{align*}
\end{proof}
\end{lemma}

\begin{proof}[Proof of Theorem~\ref{equicontthm}]
Note that $\delta(t) / t$ is decreasing for $t > 0$. Let
	$$
	g_\Dset(x) = \omega_\Dset
	\varphi\left( \frac{x - \inf \Dset}{|\Dset|} \right).
	$$
Then
	$$
	g_\Dset^{(k)} = \frac{\omega_\Dset}{|\Dset|^k}
	\varphi^{(k)}\left( \frac{x - \inf \Dset}{|\Dset|} \right)
	$$
for all $k$, and thus for all $x, y \in \closure{\Dset}$,
	\begin{align*}
	\left| g_\Dset^{(m)}(y) - g_\Dset^{(m)}(x) \right|
	&\leq
	\| \varphi^{(m + 1)} \|_\infty
	\frac{\delta_\Dset}{|\Dset|^{m + 1}} |y - x|
	=
	\| \varphi^{(m + 1)} \|_\infty
	\frac{\delta(|\Dset|)}{|\Dset|} |y - x|\\
	&\leq
	\| \varphi^{(m + 1)} \|_\infty
	\frac{\delta(|y - x|)}{|y - x|} |y - x|
	=
	\| \varphi^{(m + 1)} \|_\infty \delta(|y - x|).
	\end{align*}
Since $g_\Dset^{(m)}$ is constant on $\Dset^c$ the inequality
	$$
	\left| g_\Dset^{(m)}(y) - g_\Dset^{(m)}(x) \right|
	\leq
	\| \varphi^{(m + 1)} \|_\infty \delta(|y - x|)
	$$
then holds for all $x, y \in \rea$, and the conclusion follows by
Lemma~\ref{gensmoothlemma} since
	\begin{align*}
	f_\omega(x) &= x + \sum_{\Dset \in D} g_\Dset(x). \qedhere
	\end{align*}
\end{proof}

\section{Proof of Theorem~\ref{iterthm}}
Recall that $\Phi = \{F_1, \ldots, F_N\}$ is an
iterated function system where each $F_i$ is a contracting $C^{1 + \beta}$
diffeomorphism, that $E$ is the attractor of $\Phi$ and that
the interiors of the convex hulls of $F_i(E)$ and $F_j(E)$ are assumed to be 
disjoint when $i \neq j$.
If $\rho = \rho_1 \ldots \rho_n$ is a finite string over $\{1, \ldots, N \}$, let
$F_\rho = F_{\rho_1} \circ \ldots \circ F_{\rho_n}$ and $E_\rho = F_\rho(E)$.
By the principle of bounded distortion there is a
constant $B$ such that
	$$
	B^{-1}|E_\rho||x - y| \leq |F_\rho(x) - F_\rho(y)| \leq B|E_\rho||x - y|
	$$
for every $\rho$ and every $x, y$ in a compact ball that includes $E$
(this is proved in \cite[Proposition~4.2]{falconer97} for iterated
function systems consisting of contracting $C^2$-diffeomorphisms, and the proof
goes through with a small modification for $C^{1 + \beta}$-contractions).
In particular, there is some $\gamma \in (0, 1)$ such that
	$$
	\gamma |x - y| \leq |F_i(x) - F_i(y)|
	$$
for every $i$ and every $x, y$ in that ball.

\begin{lemma} \label{iterholelemma}
Let $\lambda$ be a probability measure on $E$ such that
	$$
	\lambda(I) \leq A|I|^s
	$$
for every interval $I$. Then there are constants $a$ and $b > 0$
such that
	$$
	\#\left\{ \Dset \in D; \, \Dset \subset J \text{ and } |\Dset| \geq x^{-1}\right\}
	\geq a + b \left( \log \lambda(J) + s \log x \right)
	$$
for every interval $J$ and every $x > 0$.

\begin{proof}
For $y > 0$, let
	$$
	\eta(y) = \# \left\{\Dset \in D; \, |\Dset| \geq y^{-1} \right\},
	$$
and fix some $\Dset_0 \in D$. Given $y$, let $n$
be the unique integer such that
	$$
	\gamma^{n + 1}|\Dset_0| < y^{-1} \leq \gamma^n|\Dset_0|.
	$$
If $n \geq 0$ then $|F_\rho(\Dset_0)| \geq y^{-1}$ for each $\rho$ of length $n$ and
there are $N^n$ such $\rho$:s, so $\eta(y) \geq N^n \geq n \log N$. If $n < 0$,
it is still true that $\eta(y) \geq n \log N$. Thus
	$$
	\eta(y) \geq n \log N
	\geq \left( \frac{\log y + \log |\Dset_0|}{-\log \gamma} - 1 \right) \log N.
	$$

Now assume that $s > 0$, since otherwise the conclusion of the lemma
holds trivially. Take an interval $J$ with positive $\lambda$-measure and some positive $x$.
Consider the sets of the form $E_\rho$ that are maximal (with respect to inclusion) subject
to the condition of having $\lambda$-measure less than or equal to $\lambda(J) / 3$.
Their union is $E$ and the intersection of two different such sets
contains at most one point. Therefore $J$ intersects at least three of them,
and must include at least one (the one in the middle). Thus there is some
$\rho$ such that $E_{\rho} \subset J$ and
	$$
	|E_{\rho}| \geq \gamma |E_{\rho'}| \geq
	\gamma \left( \frac{\lambda(E_{\rho'})}{A} \right)^{1 / s} \geq
	\gamma \left( \frac{\lambda(J)}{3A} \right)^{1 / s},
	$$
where $\rho'$ is the longest proper prefix of $\rho$. If
$\Dset \in D$ then $F_\rho(\Dset) \in D$ and $F_\rho(\Dset) \subset \conv E_\rho$
and $|F_\rho(\Dset)| \geq B^{-1} |E_\rho||\Dset|$, so it follows that
	\begin{align*}
	\#&\big\{ \Dset \in D; \, \Dset \subset J \text{ and } |\Dset| \geq x^{-1}\big\}
	\geq \# \left\{ \Dset \in D; \, \Dset \subset \conv E_\rho \text{ and }
	|\Dset| \geq x^{-1} \right\} \\
	&\geq
	\eta\left( B^{-1} |E_\rho| x \right)
	\geq
	\eta\left(B^{-1} \gamma (3 A)^{-1 / s} \lambda(J)^{1 / s} x\right) \\
	&\geq
	\left(\frac{\log\left(B^{-1} \gamma (3 A)^{-1 / s}\right) +
	\log |\Dset_0|}{-\log \gamma} - 1\right) \log N
	+\frac{\log N}{- s \log \gamma} \left( \log \lambda(J) + s\log x \right).
	\end{align*}
\end{proof}
\end{lemma}

\begin{proof}[Proof of Theorem~\ref{iterthm}]
Take any $\varepsilon > 0$. Using Lemma~\ref{iterholelemma} at the last
step, there are constants $a$ and $b > 0$ and some $x_0$ such that
if $x \geq x_0$ then
	\begin{align*}
	\psi(J, x) &=
	\#\left\{\Dset \in D; \, \Dset \subset J \text{ and }
	|\Dset|^m \delta(|\Dset|) \geq x^{-1} \right\} \\
	&\geq
	\#\left\{\Dset \in D; \, \Dset \subset J \text{ and }
	|\Dset|^{m + \alpha + \varepsilon} \geq x^{-1} \right\} \\
	&=
	\#\left\{\Dset \in D; \, \Dset \subset J \text{ and }
	|\Dset| \geq x^{-1 / (m + \alpha + \varepsilon)} \right\} \\
	&\geq
	a + b \left( \log \lambda(J) + \frac{s}{m + \alpha + \varepsilon} \log x \right),
	\end{align*}
and thus $\dimf f_\omega \lambda \geq s / (m + \alpha + \varepsilon)$ almost
surely by Theorem~\ref{thm1}. Letting $\varepsilon \to 0$ along a countable
set proves the theorem.
\end{proof}

\section{Proof of Theorem~\ref{thm2} and Theorem~\ref{endthm1}} \label{lastsection}
\begin{lemma}	\label{tlemma}
Let $F$ be a Borel subset of \rea such that $\hmeas^s(F) > 0$.
Then there exists some $t \in \rea$ and a probability measure $\lambda$
on $C \cap (F + t)$ such that
	$$
	\lambda(I) \leq A |I|^s
	$$
for some constant $A$ and every interval $I$.

\begin{proof}
Since $\hmeas^s(F) > 0$, there is a probability measure $\mu$ on $F$ such that
	$$
	\mu(I) \leq A_0 |I|^s
	$$
for some constant $A_0$ and every Borel set $I$ (see
\cite[Proposition~4.11 and Corollary~4.12]{falconer90}, and also
\cite[Theorem~48]{rogers70}). Then, using
that Lebesgue measure is invariant under translation,
	\begin{align*}
	0 < \mu(\rea) \leb(C)
	&= \iint \chi_C(t) \intd{t} \intd{\mu}(x)
	= \iint \chi_C(x + t) \intd{t} \intd{\mu}(x)\\
	&= \iint \chi_C(x + t) \intd{\mu}(x) \intd{t}
	= \int \mu_t(C) \intd{t},
	\end{align*}
where $\mu_t$ denotes the translation of $\mu$ by $t$. It follows
that there is some $t$ such that $\mu_t(C) > 0$. For that $t$, let
	$$
	\lambda = \frac{\restr{\mu_t}{C}}{\mu_t(C)}.
	$$
Then $\lambda$ is a probability measure on $C \cap (F + t)$ and
	$$
	\lambda(I) \leq \frac{A_0}{\mu_t(C)} |I|^s
	$$
for every interval $I$.
\end{proof}
\end{lemma}

\begin{lemma} \label{holelemma}
Take $E = C$ and let $\lambda$ be a probability measure on $C$ such that
	$$
	\lambda(I) \leq A |I|^s
	$$
for every interval $I$. Then there are constants $a$ and $b > 0$
and a function $\theta$, such that $\lim_{x \to \infty} \theta(x)  = 1$
and
	$$
	\# \left\{ \Dset \in D;	\, \Dset \subset J \text{ and }
	|\Dset| \geq x^{-1} \right\}
	\geq a + b(\log \lambda(J) + s \theta(x) \log x)
	$$
for every interval $J$ and every $x \geq (1 - 2c_1)^{-1}$.

\begin{proof}
Fix $J$ and $x$ and let $n$ be the unique integer such that
	$$
	(1 - 2c_{n + 2}) \prod_{i = 1}^{n + 1} c_i < x^{-1}
	\leq (1 - 2c_{n + 1}) \prod_{i = 1}^n c_i.
	$$
For any two elements of $D$ that have the same size there
is a larger element of $D$ that lies between them (with respect
to the order of \rea), so there is a unique largest $\Dset \in D$ that
intersects $J$. Then there is a connected component $J'$ of
$J \setminus \Dset$ such that $\lambda(J') \geq \lambda(J) / 2$,
and thus
	$$
	|J'| \geq \left( \frac{\lambda(J)}{2A} \right)^{1 / s}.
	$$
Let
	$$
	k = \ceil{\frac{\log(2A) - \log \lambda(J)}{s \log 2}}.
	$$
Then
	$$
	\prod_{i = 1}^k c_i \leq 2^{-k} \leq |J'|,
	$$
so $J'$ intersects a connected component of $[0, 1] \setminus C_k$.
That connected component is smaller than \Dset, and therefore
$\Dset \subset [0, 1] \setminus C_k$ as well.
It follows that $J'$ includes one of the connected components of
$C_k$, and hence there are at least $2^{n - k}$ elements of $D$ that
are included in $J$ and have size $(1 - 2c_{n + 1}) \prod_{i = 1}^n c_i$. Thus
	\begin{align*}
	\# \left\{ \Dset \in D;	\, \Dset \subset J \text{ and } |\Dset| \geq x^{-1} \right\}
	&\geq 2^{n - k}	\geq (n - k) \log 2 \\
	\geq
	-\frac{s\log 2 + \log(2A)}{s} + \frac{\log \lambda(J)}{s} &-
	\frac{n \log 2 \log x}{\log(1 - 2c_{n + 2}) + \sum_{i = 1}^{n + 1} \log c_i},
	\end{align*}
which has the desired form with
	$$
	a = -\frac{s\log 2 + \log(2A)}{s},
	\qquad
	b = \frac{1}{s}
	$$
and
	$$
	\theta(x) =	\frac{-n \log 2}%
	{\log(1 - 2c_{n + 2}) + \sum_{i = 1}^{n + 1} \log c_i}.
	$$
This proves the lemma since $n \to \infty$ when $x \to \infty$
and $c_i \to 1 / 2$ when $i \to \infty$.
\end{proof}
\end{lemma}

\begin{proof}[Proof of Theorem~\ref{thm2}]
By Lemma~\ref{tlemma} there is some $t \in \rea$ and a probability measure
$\lambda$ on $C \cap (F + t)$ such that
	$$
	\lambda(I) \leq A |I|^s
	$$
for every interval $I$. Take any $\varepsilon > 0$. Since $\lambda$ is
a probability measure on $C$, it follows from Lemma~\ref{holelemma}
that there constants $a$ and $b > 0$ and some $x_0$ such that if
$x \geq x_0$ then
	\begin{align*}
	\psi(J, x) &= \# \left\{ \Dset \in D;
	\, \Dset \subset J \text{ and }
	|\Dset|^m \delta(|\Dset|) \geq x^{-1} \right\} \\
	&\geq
	\# \left\{ \Dset \in D;
	\, \Dset \subset J \text{ and }
	|\Dset|^{m + \alpha + \varepsilon} \geq x^{-1} \right\} \\
	&=
	\# \left\{ \Dset \in D;
	\, \Dset \subset J \text{ and }
	|\Dset| \geq x^{-1 / (m + \alpha + \varepsilon)} \right\} \\
	&\geq
	a + b\left(\log \lambda(J) +
	\left(\frac{s}{m + \alpha + \varepsilon} - \varepsilon\right) \log x\right),
	\end{align*}
Thus $\dimf f_\omega \lambda \geq s / (m + \alpha + \varepsilon) - \varepsilon$
almost surely by Theorem~\ref{thm1}, and letting $\varepsilon \to 0$ along a
countable set shows that $\dimf f_\omega \lambda \geq s / (m + \alpha)$ almost surely.
\end{proof}

\begin{proof}[Proof of Theorem~\ref{endthm1}]
Let $s = \dimh F$. If $s = 0$ then the statement is trivial, so assume that $s > 0$
and let $(d_k)_{k = 1}^\infty$ be an increasing sequence of positive numbers that converges
to $s$. Then there are disjoint open intervals $\{J_k\}_{k = 1}^\infty$
such that $\hmeas^{d_k}(F \cap J_k) > 0$ for all $k$. They can be
constructed recursively along with open intervals $\{I_k\}_{k = 1}^\infty$
as follows.

Let $I_1 = \rea$. Assuming that
$I_k$ has been defined and that $\dimh F \cap I_k = s$, there is a compact
subset $F_k$ of $F \cap I_k$ such that $0 < \hmeas^{d_k}(F_k) < \infty$
(see \cite[Theorem~48]{rogers70}). Let $x_k$ be a point such that
	$$
	\hmeas^{d_k}(F_k \cap (-\infty, x_k]) = \hmeas^{d_k}(F_k \cap [x_k, \infty)).
	$$
Then $I_k \setminus \{x_k\}$ is a disjoint union of two open intervals. Choose $I_{k + 1}$
to be one of these intervals so that $\dimh(F \cap I_{k + 1}) = s$, and let $J_k$ be
the other interval. Note that $\hmeas^{d_k}(F \cap J_k) \geq \hmeas^{d_k}(F_k) / 2 > 0$.

Take $E = C$ in the construction of $f_\omega$. By Theorem~\ref{thm2}, there is for
each $k$ some $t_k \in \rea$ and $\omega_k \in \Omega$ such that
	$$
	\dimf f_{\omega_k}(C \cap (F \cap J_k + t_k)) \geq \frac{d_k}{m + \alpha}.
	$$
Let $a_k = \inf J_k \cap (C - t_k)$ and $b_k = \sup J_k \cap (C - t_k)$.
Define $g_k$ on $[a_k, b_k]$ by
	$$
	g_k(x) = (f_{\omega_k}(x + t_k) - x) - (f_{\omega_k}(a_k + t_k) - a_k),
	$$
and set $g_k$ to $0$ on $(-\infty, a_k)$ and to $g_k(b_k)$ on $(b_k, \infty)$.
By Theorem~\ref{equicontthm} the function $f_{\omega_k}$ is $m$ times
differentiable and $f_{\omega_k}^{(m)}$ is uniformly continuous
with modulus $2 \| \varphi^{(m + 1)} \|_\infty \delta$, and it is clear from
the way that $f_\omega$ was extended to \rea (see Section~\ref{extsec})
that
	$$
	f_{\omega_k}' = 1
	\quad \text{and} \quad
	f_{\omega_k}'' = \ldots = f_{\omega_k}^{(m)} = 0 \text{ on } C.
	$$
It follows that $g_k$ is $m$ times differentiable and that
$g_k^{(m)}$ is uniformly continuous with modulus
$2 \| \varphi^{(m + 1)} \|_\infty \delta$.
Thus by Lemma~\ref{gensmoothlemma},
	$$
	g(x) = \sum_{k = 1}^\infty g_k(x)
	$$
is $m$ times differentiable and $g^{(m)}(x)$ is uniformly continuous
with modulus $4 \| \varphi^{(m + 1)} \|_\infty \delta$.
Let
	$$
	f(x) = x + g(x).
	$$
For $x \in (a_n, b_n)$
	$$
	f(x) = \text{const.} + f_{\omega_n}(x + t_n),
	$$
and thus
	\begin{align*}
	&\dimf f(F) \geq \sup_n \, \dimf f(F \cap (a_n, b_n)) \geq \sup_n \frac{d_n}{m + \alpha}
	= \frac{s}{m + \alpha}.
	\qedhere
	\end{align*}
\end{proof}

\bibliographystyle{plain}
\bibliography{references}
\end{document}